\documentclass{amsart}

\usepackage{amsmath}
\usepackage{enumerate}
\usepackage{amssymb}
\usepackage{amsfonts}
\newtheorem{theorem}{Theorem}
\newtheorem{proposition}{Proposition}[section]

\newcommand*{\CC}{\mathbb{C}}

\begin{document}

\title[Short proof that the Euclidean algorithm is Gaussian]{A short proof that the number of division steps in the Euclidean algorithm is normally distributed}
\author{Ian D. Morris}

\begin{abstract}D. Hensley showed in 1994 that the number of steps taken by the Euclidean algorithm to find the greatest common divisor of two natural numbers less than or equal to $n$ follows a normal distribution in the limit as $n$ tends to infinity. V. Baladi and B. Vall\'ee subsequently gave an alternative proof for both the classical Euclidean algorithm and several of its close variants, based on a detailed investigation of spectral properties of the transfer operator associated to the Gauss map, building on deep results of D. Dolgopyat. In this article we give a much shorter, albeit less quantitative, proof of this result using only basic spectral properties of the transfer operator together with the method of moments and a Tauberian theorem due to H. Delange.

MSC codes: Primary 11A05, 11K50, 68W40; Secondary 37D20, 37D35.
\end{abstract}
\maketitle

\section{Introduction}\label{Introd}

The Euclidean algorithm for determining the greatest common divisor of two natural numbers is believed to be the oldest nontrivial algorithm which remains in common use to the present day \cite[p.335]{Knuth}. By the late $16^{\mathrm{th}}$ century it had been noticed that the slowest case of the algorithm is that in which the input is a pair of successive Fibonacci numbers \cite{Sch}. Rigorous analysis of the number of division steps required by the algorithm began in earnest in the $19^{\mathrm{th}}$ century when the first logarithmic upper bound was given by P.-J.-\'E. Finck (\cite{Finck}, for a detailed historical overview see \cite{Shallit}). Analysis of the number of division steps \emph{typically} required to execute the algorithm did not develop until the second half of the $20^{\mathrm{th}}$ century, when H. Heilbronn \cite{Heilbronn} and J. D. Dixon \cite{Dixon} independently showed that the average number of divisions required to process a pair of integers both bounded by $n$ is asymptotic to $\frac{12}{\pi^2}\log 2\log n$ in the limit as $n \to \infty$. A further milestone was achieved by D. Hensley \cite{Hensley} who in 1994 showed that the deviation from this mean is asymptotically Gaussian for large $n$, giving also a local limit theorem and an upper bound for the rate of convergence to the limit distribution.

Analyses of variant versions of the Euclidean algorithm were undertaken in the late  $20^{\mathrm{th}}$ century using somewhat disparate methods  \cite{Brent,Rieger,YK}. B. Vall\'ee \cite{Val} subsequently unified the analysis of the average number of division steps of a range of Euclidean-type algorithms into a single technique based on the thermodynamic formalism of the Gauss-Kuzman-Wirsing operator (also called the \emph{density transformer} or \emph{transfer operator}) and its close relatives. This approach was extended in joint work with V. Baladi (\cite{BV}, see also \cite{BVcorr}) where Hensley's result was generalised to prove that the number of division steps asymptotically follows a Gaussian distribution throughout a class of variant Euclidean algorithms, with detailed estimates on the rate of convergence. Baladi and Vall\'ee's results also generalised Hensley's result to different cost measures of the algorithm including the binary length of the divisor and the number of occurences of particular quotients.

Both the proof given by Hensley \cite{Hensley} and that given by Baladi-Vall\'ee \cite{BV} are long and technically involved. Baladi and Vall\'ee's treatment in particular rests on a delicate investigation of the norm of certain complex perturbations of the transfer operator using techniques developed by D. Dolgopyat \cite{Dolg} in the application of thermodynamic formalism to geodesic flows. This may lead the reader to wonder whether the fact that the number of steps obeys a Gaussian distribution is amenable to a shorter proof. In this article we will show that the existence of a Gaussian limit distribution for the number of steps (but without information on the rate of convergence) can be recovered via a relatively simple strategy using only qualitative features of the spectrum of the transfer operator: indeed, the properties which we require do not significantly extend beyond those used in Vall\'ee's earlier article \cite{Val}. We also hope that the approach presented here may facilitate the generalisation of this type of result to related contexts such as \cite{LTDH,Mo14,Val08} where the corresponding Dolgopyat-type estimate is unavailable or may even be impossible. 

\section{Statement of results}\label{Statem}

The standard Euclidean algorithm, starting from a pair of integers $(u,v)$ with $0<u \leq v$, iterates the following step, which we call the division step: map the pair $(u,v)$ to the pair $(r,u)$, where $v=qu+r$ and $0 \leq r <v$. This procedure is repeated until a pair of the form $(0,k)$ is obtained, whereupon the integer $k$ is returned as the GCD of $u$ and $v$. The algorithm may be modelled via the following parallel procedure: given a rational number $x:=u/v \in (0,1]$, we map that number to the new rational number $T(x):=\frac{1}{x}-\lfloor \frac{1}{x}\rfloor$, repeating until zero is obtained. Since at each stage the integer $r=v-qu$ equals precisely $v-u\lfloor v/u\rfloor$, it is clear that both procedures involves the same number of steps when applied to each pair $(u,v)$. To measure the number of steps taken to reduce the pair $(u,v)$ is thus equivalent to finding the least integer $N(u/v) \geq 1$ such that $T^{N(u/v)}(u/v)=0$. As in the work of Vall\'ee and Baladi, in addition to studying the integer $N(u/v)$ we are easily able to adapt our arguments to the study of more general ``cost'' measurements, which assign to the pair $(u,v)$ a cost $C(u/v)$ depending on the particular sequence of divisions performed -- or, equivalently, on the digits $a_i \geq 1$ arising in the finite continued fraction representation $u/v=[a_1,\ldots,a_n]$. In principle our method is capable of studying the case in which the cost $c(n)$ associated to division by the integer $n$ increases logarithmically with $n$, but in this note we have restricted our attention to bounded cost functions $c \colon \mathbb{N} \to [0,+\infty)$ so as to simplify the exposition.

Let $\xi \colon [0,1]\to \mathbb{R}$ denote the function $\xi(x):=\frac{1}{(\log 2)(1+x)}$ which is the density of the unique absolutely continuous $T$-invariant probability measure on $[0,1]$. (In subsequent sections we shall also write $\xi$ for the extension of that function to a certain disc in the complex plane.) For each $n \geq 0$ let us define $\tilde{\Omega}_n$ to be the set of all pairs of integers $(u,v)\in\mathbb{N}^2$ such that $1 \leq u < v \leq n$, and let $\Omega_n$ denote the set of all $(u,v) \in\Omega_n$ such that $\mathrm{gcd}(u,v)=1$. Let $\mathbb{P}_n$ (resp. $\tilde{\mathbb{P}}_n$) denote the uniform probability measure on $\Omega_n$ (resp. $\tilde{\Omega}_n$). We prove the following result:
\begin{theorem}\label{qfwfq}
Let $c \colon \mathbb{N} \to [0,+\infty)$ be a bounded function which is not identically zero, and for each $u/v=[a_1,\ldots,a_n] \in (0,1) \cap \mathbb{Q}$ with $a_n \geq 2$ define $C(u/v):=\sum_{k=1}^n c(a_k)$. Then there exist constants $\mu,\sigma^2>0$ such that for all $t \in \mathbb{R}$
\[\lim_{n \to \infty} \mathbb{P}_n\left(\left\{(u,v) \in \Omega_n \colon \frac{ C(u/v)-\mu\log n}{\sigma\sqrt{\log n}} \leq t\right\}\right) = \frac{1}{\sqrt{2\pi}}\int_{-\infty}^t e^{-\frac{x^2}{2}}dx,\]
and similarly for $\tilde{\mathbb{P}}_n$ and $\tilde{\Omega}_n$ with the same constants. The mean $\mu$ satisfies
\begin{equation}\label{mean}\mu:=\frac{2}{\mathfrak{h}(T)}\sum_{n=1}^\infty c(n)\int_{\frac{1}{n+1}}^{\frac{1}{n}}\xi(x)dx\end{equation}
where $\mathfrak{h}(T):=\int_0^1\log|T'(x)|\xi(x)dx=\frac{\pi^2}{6\log2}$. If we define $\psi(x):=c(n)+\mu\log x$ for all $x \in (\frac{1}{n+1},\frac{1}{n}]$ and $\psi(0):=0$, then the variance $\sigma^2$ satisfies
\[\sigma^2=\frac{2}{\mathfrak{h}(T)}\left(\lim_{n \to \infty} \frac{1}{n}\int_0^1\left(\sum_{k=0}^{n-1} \psi(T^kx)\right)^2\xi(x)dx\right)>0.\]
\end{theorem}
As an example of the application of Theorem \ref{qfwfq}, if $c$ is set to be the constant function $1$ then the cost $C(u/v)$ corresponds to the number of division steps required to evaluate $\mathrm{gcd}(u,v)$. Alternatively, if we define $c(m):=1$ and $c(n):=0$ otherwise, then the resulting cost $C(u/v)$ is the number of occurences of the digit $m$ in the continued fraction representation of the number $u/v$.

By a classical theorem of V. A. Rokhlin \cite{Roh} the quantity $\mathfrak{h}(T)$ is precisely the entropy of the transformation $T$ with respect to the absolutely continuous invariant probability measure on $[0,1]$ with density $\xi$. However, this identification has no bearing on our argument beyond its influence on our choice of notation. While our exposition in this article restricts itself to the classical Euclidean algorithm, it is easy to modify our method along the lines of \cite{BV,Val} so as to apply to its Odd and Centred variants.

The proof of Theorem \ref{qfwfq} which we present uses the \emph{method of moments}. By standard results of probability theory, to prove that the distribution of $C(u/v)$ converges to a normal distribution it is sufficient to prove that for every integer $p \geq 0$ the sequence of centred $p^{\mathrm{th}}$ moments,
\begin{equation}\label{shirley}\frac{1}{\#\Omega_n}\sum_{(u,v) \in \Omega_n} \left(\frac{C(u/v)-\mu\log n}{\sigma\sqrt{\log n}}\right)^p,\end{equation}
converges to the $p^{\mathrm{th}}$ moment of a standard normal distribution (see e.g. \cite[\S30]{Bill}). To study these moments we consider the bivariate Dirichlet series
\[D(s,\omega):=\sum_{n =1}^\infty \frac{1}{n^{2s}} \sum_{(u,v) \in \Omega_n \setminus \Omega_{n-1}} \exp\left(\omega\left(C(u/v)-\mu\log v\right)\right).\]
The sum which appears in the $p^{\mathrm{th}}$ moment \eqref{shirley} resembles the sum of the first $n$ terms of the $p^{\mathrm{th}}$ partial derivative with respect to $\omega$ of $D(s,\omega)$ evaluated at $s=\omega=0$, differing in the presence of $\mu\log n$ versus $\mu\log v$ inside the summation. To exploit this fact we extract this sum from $D(s,\omega)$ using a Tauberian theorem due to H. Delange, stated below as Theorem \ref{you-go-tell-dr-dre-that-man-will-fuck-you-up}. When $p$ is even Delange's result may be applied directly to the calculation of the moment. For odd $p$ our analysis is complicated by the fact that the summands fail to be non-negative, and to overcome this issue we adopt a strategy due to H.-K. Hwang and S. Janson \cite{HJ,HJcorr}. These parts of the proof constitute \S\ref{Thetau} below.

To show that the Dirichlet series has the properties needed to apply Delange's theorem we equate the series with a summation over all of the possible compositions of inverse branches of $T$.  This is achieved by expressing these sums over inverse branches in terms of the Gauss-Kuzmin-Wirsing operator, the necessary features of which are studied in \S\ref{Thetra} below. In \S\ref{Thedir} we effect the translation of the properties of the transfer operator into those of the Dirichlet series.

In preparing this article we have taken some effort to minimise the prior knowledge of ergodic theory and thermodynamic formalism required  on the part of the reader. In particular, the only significant ``black box'' results to which we appeal are Delange's Tauberian theorem, the efficacy of the method of moments, and certain results from the perturbation theory of compact operators, for the last of which we refer the reader to the book by T. Kato \cite{Kato}. The decision to restrict our attention to bounded costs is partly based in this effort at conciseness: to treat the general case would require either appeal to an external result in thermodynamic formalism (as occurs in \cite{BV}) or an exposition which would have substantially lengthened \S\ref{Thetra}.

\section{Part I: the transfer operator}\label{Thetra}

In this section we define a weighted version of the Gauss-Kuzmin-Wirsing operator and establish those spectral properties of that operator which will be used in the subsequent sections. This material is to a significant extent expository and is included for the sake of completeness.

Here and throughout the article we let $\mathbb{D}:=\{z \in \mathbb{C} \colon |z-\frac{2}{3}|<1\}$, and let $\xi\colon\mathbb{D}\to\mathbb{C}$ be the holomorphic extension of the function $\xi$ defined in the introduction.
Since the function $z \mapsto z+n$ has no zeros in $\mathbb{D}$ we may unambiguously write $(z+n)^s =\exp(s\log(z+n))$ for every $z \in \mathbb{D}$ and $s \in \mathbb{C}$, where $\log (z+n)$ indicates the branch of the complex logarithm which coincides with the real logarithm when $z \in [0,1]$. We recall that a function from an open subset of $\mathbb{C}$ to a Banach space $\mathfrak{X}$ is called holomorphic if it is locally equal to a convergent power series with coefficients  in $\mathfrak{X}$. A function from an open subset of $\mathbb{C}$ to $\mathfrak{X}$ is holomorphic if and only if its composition with each element of $\mathfrak{X}^*$ is holomorphic in the ordinary sense. Throughout this article we shall repeatedly appeal without comment to the fact that a uniform limit of holomorphic functions is holomorphic, which follows from the combination of Cauchy and Morera's theorems. We recall that $H^\infty(\mathbb{D})$ denotes the set of all bounded holomorphic functions on $\mathbb{D}$, which by the aforementioned principle is a complex Banach space when equipped with the uniform norm.

Here and throughout the article we let $|c|_\infty$ denote the least upper bound of the cost function $c$, and let $\mu >0$ be as defined by \eqref{mean}. For each $s,\omega \in \mathbb{C}$ with $\Re(s)>\frac{1}{2}$, and each $f \in H^\infty(\mathbb{D})$, let us define a pair of functions $\mathcal{L}_{s,\omega}f, \mathcal{F}_{s,\omega}f \colon \mathbb{D} \to \mathbb{C}$ by
\begin{equation}\label{TO}\left(\mathcal{L}_{s,\omega}f\right)(z) = \sum_{n=1}^\infty \frac{e^{\omega c(n)}}{(z+n)^{2s+\mu\omega}}f\left(\frac{1}{z+n}\right),\end{equation}
\[\left(\mathcal{F}_{s,\omega}f\right)(z) = \sum_{n=2}^\infty \frac{e^{\omega c(n)}}{(z+n)^{2s+\mu\omega}}f\left(\frac{1}{z+n}\right).\]
(The difference between these two functions -- that the sum defining $\mathcal{F}_{s,\omega}$ does not include the index $n=1$ -- reflects the fact that the continued fraction expansion of a rational number is defined unambiguously only if the final digit is not permitted to equal $1$.)  We begin the proof of Theorem \ref{qfwfq} by quickly proving the following result which is of a somewhat classic type; compare for example \cite[Prop.1]{Mayer}, \cite[Thm. 1]{Val}, \cite[Prop. 0]{BV}.
\begin{proposition}\label{TheOne}
The equation \eqref{TO} defines a family of linear operators $\mathcal{L}_{s,\omega}$, $\mathcal{F}_{s,\omega}$ on $H^\infty(\mathbb{D})$ such that:
\begin{enumerate}[(i)]
\item
For every $(s,\omega) \in \CC$ with $\Re(s)>\frac{1}{2}$ and $|\omega|$ sufficiently small, $\mathcal{L}_{s,\omega}$ and $\mathcal{F}_{s,\omega}$ are well-defined compact operators on $H^\infty(\mathbb{D})$. The dependence of these operators on $(s,\omega)$ is holomorphic.
\item
The operator $\mathcal{L}_{1,0}$ acts boundedly on $L^1([0,1])$ and satisfies $\int_0^1 (\mathcal{L}_{1,0}f)(x)dx=\int_0^1 f(x)dx$ for all $f \in L^1([0,1])$. If $f, g \colon [0,1] \to \mathbb{C}$ are such that $f$ and $f(g\circ T)$ both belong to $L^1([0,1])$, then $\mathcal{L}_{1,0}(f(g \circ T)) = (\mathcal{L}_{1,0}f)g$ Lebesgue almost everywhere on $[0,1]$. If $f \in L^1([0,1])$ then also $f \circ T \in L^1([0,1])$ and $|f\circ T|_{L^1} \leq C|f|_{L^1}$ where $C>1$ is constant.
\item
There exists an open set $\mathcal{V} \subset \CC^2$ containing $(1,0)$ such that for all $(s,\omega) \in \mathcal{V}$ we may write $\mathcal{L}_{s,\omega} = \lambda(s,\omega)\mathcal{P}_{s,\omega} + \mathcal{N}_{s,\omega}$ where $\mathcal{P}_{s,\omega}\mathcal{N}_{s,\omega}=\mathcal{N}_{s,\omega}\mathcal{P}_{s,\omega}=0$, $\mathcal{P}_{s,\omega}$ is a projection onto a one-dimensional subspace of $H^\infty(\mathbb{D})$ which is an eigenspace of $\mathcal{L}_{s,\omega}$, $\mathcal{N}_{s,\omega}$ has spectral radius strictly less than $1$, and $\lambda$ is a complex number. The dependence of $\lambda, \mathcal{P}$ and $\mathcal{N}$ on $(s,\omega)$ is holomorphic throughout $\mathcal{V}$. We have $\lambda(1,0)=1$ and $\mathcal{P}_{1,0}f = \left(\int_0^1 f(x)dx\right) \xi$ for all $f \in H^\infty(\mathbb{D})$.
j\end{enumerate}
\end{proposition}

\begin{proof}
(i). For each $n \geq 1$ and $f \in H^\infty(\mathbb{D})$ it is easily seen that the equation
\[\left(\mathcal{B}_{s,\omega,n}f\right)(z) := \frac{e^{\omega c(n)}}{\left(n+z\right)^{2s+\mu\omega}}f\left(\frac{1}{n+z}\right)\]
defines a holomorphic function on $\mathbb{D}$. Since $|z^u|=e^{-\Im(u)\arg(z)}|z|^{\Re(u)}$ when $z,u \in \mathbb{C}$ with $\Re(z)>0$,  we furthermore have
\begin{equation}\label{eq:opnormbound}|\mathcal{B}_{s,\omega,n}f|_\infty \leq \left(\frac{2}{n}\right)^{2\Re(s)+\mu\Re(\omega)}e^{|c|_\infty\Re(\omega)+\pi \left(\Im(s)+\frac{\mu}{2}\Im(\omega)\right)}|f|_\infty.\end{equation}
In particular $\mathcal{B}_{s,\omega,n}$ is a bounded linear operator on $H^\infty(\mathbb{D})$, and it is clear that $\mathcal{B}_{s,\omega,n}$ depends holomorphically on $(s,\omega)$.
Since $\mathcal{L}_{s,\omega}$ is simply the sum over all $n \geq 1$ of the operators $\mathcal{B}_{s,\omega,n}$, and for $\Re(s)>\frac{1}{2}$ and $|\omega|$ small enough the inequality \eqref{eq:opnormbound} implies that the convergence of this sum is locally uniform with respect to $(s,\omega)$, $\mathcal{L}_{s,\omega}$ is a holomorphic family of bounded operators on $H^\infty(\mathbb{D})$. Taking the sum over $n \geq 2$ shows the family $\mathcal{F}_{s,\omega}$ to have the same properties.

To complete the proof of (i) it suffices to show that each operator $\mathcal{B}_{s,\omega,n}$ is compact. If $h \colon \mathbb{D} \to \mathbb{D}$ is any of the linear fractional transformations $z \mapsto 1/(z+n)$ then we may expand its domain of definition to the larger disc $\hat{\mathbb{D}}:=\{z \in \mathbb{C}\colon |z-\frac{2}{3}|<\frac{17}{16}\}$, and note that $h(\hat{\mathbb{D}})$ is contained in a compact subset of $\mathbb{D}$.  The elements of the set $\{f \circ h \colon f \in H^\infty(\mathbb{D}), |f|_\infty \leq 1\}$ may therefore be simultaneously extended to the larger domain $\hat{\mathbb{D}}$ whilst remaining uniformly bounded and holomorphic. It follows by Montel's compactness principle that this set is a precompact subset of $H^\infty(\mathbb{D})$ and hence the linear operator $f \mapsto f \circ h$ on $H^\infty(\mathbb{D})$ is compact. If $g \in H^\infty(\mathbb{D})$ is any bounded function then the map $f \mapsto g \cdot (f \circ h)$ is the composition of the compact operator $f \mapsto f \circ h$ with the bounded operator $f \mapsto g\cdot f$ and thus is a compact operator on $H^\infty(\mathbb{D})$ with norm $|g|_\infty$. By this principle it follows that each of the operators $\mathcal{B}_{s,\omega,n}$ is compact.

(ii). For all $f \in L^1([0,1])$, by taking the substitution $y=1/(n+x)$ we obtain
\[\int_0^1\! \left(\mathcal{L}_{1,0}f\right)(x)dx = \int_0^1\sum_{n=1}^\infty \frac{1}{(x+n)^2}f\left(\frac{1}{x+n}\right)dx = \sum_{n=1}^\infty \int_{\frac{1}{n+1}}^{\frac{1}{n}}\! f(y)\,dy = \int_0^1\! f(x)dx\]
and an obvious modification shows that $|\mathcal{L}_{1,0}f|_{L^1} \leq |f|_{L^1}$. If $f \in L^1([0,1])$, $g \colon [0,1] \to \mathbb{C}$ is measurable and $x \in [0,1]$ then
\begin{align*}(\mathcal{L}_{1,0}(f\cdot (g \circ T)))(x)&=\sum_{n=1} \frac{1}{(x+n)^2} f\left(\frac{1}{x+n}\right)g\left(T\left(\frac{1}{n+x}\right)\right)\\
&= g(x)\sum_{n=1}^\infty\frac{1}{(x+n)^2}f\left(\frac{1}{x+n}\right)=g(x)(\mathcal{L}_{1,0}f)(x)\end{align*}
so that $\mathcal{L}_{1,0}((g \circ T)f)=g(\mathcal{L}_{1,0}f)$ almost everywhere as claimed. In particular if $f \in L^1([0,1])$ is essentially bounded then clearly $f \circ T \in L^1([0,1])$ and we obtain $\int_0^1|f (Tx)|dx=\int_0^1\mathcal{L}_{1,0}(|f \circ T|)(x)dx \leq  |\mathcal{L}_{1,0}\mathbf{1}|_\infty \int_0^1|f(x)|dx$, which proves the last part of (ii) for $f$; the general case follows by approximation.

(iii). Since $\mathcal{L}_{1,0}$ is compact the nonzero elements of its spectrum are all isolated eigenvalues of finite multiplicity. By direct calculation $\mathcal{L}_{1,0}\xi=\xi$. 
Let $\hat\xi$ be an eigenfunction of $\mathcal{L}_{1,0}$ corresponding to some eigenvalue $\lambda$, and by rescaling if necessary suppose that $\sup_{x \in [0,1]} |\hat\xi(x)|/\xi(x)=1$ with this supremum attained at $x_0$, say. We have
\begin{align*}|\lambda\hat\xi(x_0) |=|\mathcal{L}_{1,0}\hat\xi(x_0)|&=\left|\sum_{n = 1}^\infty\frac{1}{(x_0+n)^2}\hat\xi\left(\frac{1}{x_0+n}\right)\right| \leq \sum_{n=1}^\infty\frac{1}{(x_0+n)^2} \left|\hat\xi\left(\frac{1}{x_0+n}\right)\right| \\
&\leq  \sum_{n=1}^\infty\frac{1}{(x_0+n)^2} \xi\left(\frac{1}{x_0+n}\right) =\left(\mathcal{L}_{1,0}\xi\right)(x_0)=\xi(x_0)=|\hat\xi(x_0)|\end{align*}
and therefore $|\lambda| \leq 1$, so $\mathcal{L}_{1,0}$ has spectral radius $1$. If $|\lambda|=1$ then the above inequalities must be equations. Since $|\hat\xi(x)|\leq \xi(x)$ for all $x \in [0,1]$ we deduce that $|\hat\xi(1/(x_0+n))|=\xi(1/(x_0+n))$ for all $n \geq 1$, since otherwise the second inequality would be strict; and moreover the argument of $\hat\xi(1/(x_0+n))$ must be independent of $n$ since otherwise the first inequality would be strict. It follows that $\hat\xi(1/(x_0+n))\xi(1/(x_0+n))^{-1}$ is constant in $n$, and since $1/(x_0+n) \to 0 \in \mathbb{D}$ as $n \to \infty$ we conclude that the holomorphic function $\hat\xi / \xi$ must be constant on $\mathbb{D}$. We have thus seen that $\ker(\mathcal{L}_{1,0}-\mathrm{Id}_{H^\infty(\mathbb{D})})$ is one-dimensional, that $\rho(\mathcal{L}_{1,0})=1$, and that $\mathcal{L}_{1,0}$ has no other eigenvalues of modulus one. To see that $\mathcal{L}_{1,0}$ has no generalised eigenfunctions at $1$ we observe that if $\hat\xi \in \ker (\mathcal{L}_{1,0}-\mathrm{Id}_{H^\infty(\mathbb{D})})^2$ then $(\mathcal{L}_{1,0}-\mathrm{Id}_{H^\infty(\mathbb{D})})\hat\xi$ must be proportional to $\xi$, but by part (ii) this function has zero integral on $[0,1]$ and hence must be zero. We conclude that $ \ker (\mathcal{L}_{1,0}-\mathrm{Id}_{H^\infty(\mathbb{D})})^2= \ker (\mathcal{L}_{1,0}-\mathrm{Id}_{H^\infty(\mathbb{D})})$ and $1$ is a simple eigenvalue of $\mathcal{L}_{1,0}$ acting on $H^\infty(\mathbb{D})$.

Let $\Gamma\subset \mathbb{C}$ be an anticlockwise-oriented closed curve which encloses $1$ but does not enclose any other point of the spectrum of $\mathcal{L}_{1,0}$. By \cite[p.212]{Kato} we may find an open ball $\mathcal{V}$ containing $(1,0)$ such that the spectrum of $\mathcal{L}_{s,\omega}$ does not intersect $\Gamma$ for any $(s,\omega) \in \mathcal{V}$. Let $\mathcal{P}_{s,\omega}:=-\frac{1}{2\pi i}\int_\Gamma (z\cdot\mathrm{Id}_{H^\infty(\mathbb{D})}-\mathcal{L}_{s,\omega})^{-1}dz$ be the Riesz projection associated to the curve $\Gamma$, which is a projection commuting with $\mathcal{L}_{s,\omega}$ (\cite[p.178]{Kato}) and depends holomorphically on $(s,\omega)$ (\cite[p.369]{Kato}). The image of $\mathcal{P}_{1,0}$ is a closed subspace since that operator is a continuous projection, and is preserved by $\mathcal{L}_{1,0}$ since that operator commutes with $\mathcal{P}_{1,0}$. The spectrum of $\mathcal{L}_{1,0}$ restricted to the image of $\mathcal{P}_{1,0}$ is precisely $\{1\}$ (by \cite[p.178]{Kato}) and in particular the restriction of $\mathcal{L}_{1,0}$ to that space is both compact and invertible. The image of $\mathcal{P}_{1,0}$ is therefore finite-dimensional, and since $1$ is a simple eigenvalue of $\mathcal{L}_{1,0}$ it follows that $\mathrm{rank}\, \mathcal{P}_{1,0}=1$. By \cite[p.212]{Kato} we thus have $\mathrm{rank}\, \mathcal{P}_{s,\omega} = 1$ for all $(s,\omega) \in \mathcal{V}$. The function $\mathcal{P}_{s,\omega}\xi$ is thus an eigenfunction of $\mathcal{L}_{s,\omega}$ corresponding to some eigenvalue $\lambda(s,\omega) \in \mathbb{C}$. Define $\mathcal{N}_{s,\omega}:=\mathcal{L}_{s,\omega}(\mathrm{Id}_{H^\infty(\mathbb{D})}-\mathcal{P}_{s,\omega})$ to obtain $\mathcal{L}_{s,\omega}=\lambda(s,\omega)\mathcal{P}_{s,\omega}+\mathcal{N}_{s,\omega}$. The spectrum of $\mathcal{N}_{1,0}$ equals the spectrum of $\mathcal{L}_{1,0}$ with $1$ removed (\cite[p.178]{Kato}) so the spectral radius of $\mathcal{N}_{1,0}$ is less than one, and by upper semicontinuity of the spectral radius we may replace $\mathcal{V}$ with a smaller set to obtain $\rho(\mathcal{N}_{s,\omega})<1$ for all $(s,\omega) \in \mathcal{V}$.

For each $f \in H^\infty(\mathbb{D})$ we have $\mathcal{L}_{1,0}\mathcal{P}_{1,0}f=\mathcal{P}_{1,0}f$ so that $\mathcal{P}_{1,0}f$ is proportional to $\xi$, and since $\rho(\mathcal{N}_{1,0})<1$ we have $\int_0^1(\mathcal{P}_{1,0}f)(x)dx=\lim_{n \to \infty}\int_0^1(\mathcal{L}_{1,0}^nf)(x)dx=\int_0^1f(x)dx$ using (ii), which shows that $\mathcal{P}_{1,0}f=(\int_0^1f(x)dx)\xi$ as required. Since in particular $\mathcal{P}_{1,0}\xi=\xi$ it follows that $(\mathcal{P}_{s,\omega}\xi)(0)\neq 0$ when $(s,\omega)$ is close to zero, so by taking $\mathcal{V}$ smaller if necessary we may write $\lambda(s,\omega)=((\mathcal{P}_{s,\omega }\xi)(0))^{-1}(\mathcal{L}_{s,\omega}\mathcal{P}_{s,\omega }\xi)(0)$ for all $(s,\omega) \in \mathcal{V}$ which shows that $\lambda$ is holomorphic.
\end{proof}
The elimination of unit eigenvalues on the line $\Re(s)=1$ is classical when transfer operators interact with Tauberian arguments, see for example \cite[Prop. 6.1]{PP}. For versions of this argument in the context of transfer operators associated to number-theoretic algorithms, see for example \cite[Lemma 6.15]{Mo14},\cite[\S5.5]{Val}.
\begin{proposition}\label{they-said-i-got-too-much-flow}
Let $s \in \CC$ such that $\Re(s) \geq 1$ and $s \neq 1$. Then $\rho(\mathcal{L}_{s,0})<1$.
\end{proposition}
\begin{proof}
Let us first show that if $\Re(s) > 1$ then $\rho(\mathcal{L}_{s,0}) <1$. Using Proposition \ref{TheOne}(i) we may choose an eigenfunction $\xi_s$ of $\mathcal{L}_{s,0}$ which corresponds an eigenvalue of modulus $\rho(\mathcal{L}_{s,0})$. By rescaling this function if necessary we assume that $\sup_{x\in [0,1]} |\xi_{s}(x)|/\xi(x)=1$ with this supremum attained at some point $x_0 \in [0,1]$. For every $n \geq 1$ we have $0<(x_0+n)^{-2}<1$ and therefore $|(x_0+n)^{-2s}|<(x_0+n)^{-2}$. Thus
\begin{align*}\rho(\mathcal{L}_{s,0})\left|\xi_{s}(x_0)\right|&=\left|\mathcal{L}_{s,0}\xi_{s}(x_0)\right|\leq\sum_{n=1}^\infty \left|\frac{1}{(x_0+n)^{2s}} \xi_{s}\left(\frac{1}{x_0+n}\right)\right|\\ &< \sum_{n=1}^\infty \frac{1}{(x_0+n)^2}\xi\left(\frac{1}{x_0+n}\right)
= \left(\mathcal{L}_{1,0}\xi\right)(x_0) = \xi(x_0) =|\xi_{s}(x_0)|\end{align*}
and therefore $\rho(\mathcal{L}_{s,0})<1$ as claimed. 

We now vary this argument to show that $\rho(\mathcal{L}_{1+it,0})<1$ for all nonzero real numbers $t$. Given $t\in\mathbb{R}$, let $\xi_t$ be an eigenfunction of $\mathcal{L}_{1+it,0}$ which corresponds to an eigenvalue of modulus $\rho(\mathcal{L}_{1+it,0})$. In the same manner we rescale $\xi_t$ so that $\sup_{x\in[0,1]} |\xi_t(x)|/\xi(x)=1$ and choose $x_0 \in [0,1]$ attaining this supremum. We similarly estimate
\begin{align*}\rho(\mathcal{L}_{1+it,0})|\xi_t(x_0)|&=|\mathcal{L}_{1+it,0}\xi_{t}(x_0)|=\left|\sum_{n=1}^\infty \frac{1}{(x_0+n)^{2+2it}} \xi_t\left(\frac{1}{x_0+n}\right)\right|\\
&\leq\sum_{n=1}^\infty \left|\frac{1}{(x_0+n)^{2+2it}} \xi_t\left(\frac{1}{x_0+n}\right)\right| \leq\sum_{n=1}^\infty \frac{1}{(x_0+n)^{2}} \xi\left(\frac{1}{x_0+n}\right)\\
&= \left(\mathcal{L}_{1,0}\xi\right)(x_0) = \xi(x_0) = |\xi_t(x_0)|\end{align*}
which in particular establishes $ \rho(\mathcal{L}_{1+it,0})\leq 1$. Let us suppose that $\rho(\mathcal{L}_{1+it,0})=1$, in which case both inequalities above are equations. As with the proof of Proposition \ref{TheOne}(i) this is only possible if the argument of $(x_0+n)^{-2-2it} \xi_t(1/(x_0+n))$ is independent of $n$ and the modulus of that expression is equal to $(x_0+n)^{-2}\xi(1/(x_0+n))$. It follows that for some fixed $\theta\in\mathbb{R}$ and for all $n \geq 1$
\[\frac{e^{i\theta}}{(x_0+n)^{2+2it}}\xi_t\left(\frac{1}{x_0+n}\right)=\frac{1}{(x_0+n)^2}\xi\left(\frac{1}{x_0+n}\right)\]
and therefore
\[\lim_{n \to \infty}\frac{e^{i\theta}}{(x_0+n)^{2it}}=\lim_{n \to \infty} \frac{\xi\left(\frac{1}{x_0+n}\right)}{\xi_t\left(\frac{1}{x_0+n}\right)}=\frac{\xi(0)}{\xi_t(0)},\]
but the former sequence is divergent if $t \neq 0$. We conclude that if $\rho(\mathcal{L}_{1+it,0})=1$ then necessarily $t=0$, and this completes the proof of the proposition.
\end{proof}

\begin{proposition}\label{a-owl}
Let $\psi \colon [0,1] \to \mathbb{C}$ be as defined in Theorem \ref{qfwfq}. The function $\lambda$ satisfies
\[\frac{\partial \lambda}{\partial s}(1,0) = -\mathfrak{h}(T)<0,\qquad\frac{\partial \lambda}{\partial \omega}(1,0) = 0,\]
\[\frac{\partial^2 \lambda}{\partial \omega^2}(1,0)= \lim_{n \to \infty} \frac{1}{n}\int_0^1\left(\sum_{k=0}^{n-1} \psi(T^kx)\right)^2\xi(x)dx>0.\]
\end{proposition}
\begin{proof}
Let us write $\xi_{s,\omega}:=\mathcal{P}_{s,\omega}\xi$ when $(s,\omega) \in \mathcal{V}$ so that $\mathcal{L}_{s,\omega}\xi_{s,\omega}=\lambda(s,\omega)\xi_{s,\omega}$ in this region. Since $\mathcal{P}_{1,0}\xi=\xi$ we have $\xi_{1,0}=\xi$. Since $\mathcal{P}_{s,\omega}$ is holomorphic, $\xi_{s,\omega}$ is infinitely differentiable in $H^\infty(\mathbb{D})$ and hence also in $L^1([0,1])$: we use the notation $\xi^s$,$\xi^\omega$ to refer to the first partial derivatives with respect to $s$ and $\omega$, and $\xi^{\omega\omega}$ for the second derivative with respect to $\omega$. We denote the corresponding derivatives of $\lambda$ by $\lambda_s$,$\lambda_\omega$ and $\lambda_{\omega\omega}$. 

We claim that $s \mapsto |T'|^{1-s}\xi_{s,0}$ is a holomorphic map from a small neighbourhood of $1$ into $L^1([0,1])$. By the holomorphicity of $s \mapsto \xi_{s,0}$ near $s=1$ we may write $\xi_{s,0}=\sum_{n=0}^\infty (s-1)^ng_n$ where each $g_n \in H^\infty(\mathbb{D})$ and $\limsup_{n \to \infty}|g_n|_\infty^{1/n}<\infty$. On the other hand $|T'(x)|^{1-s}=x^{2-2s}=\sum_{n=0}^\infty (\log x)^n(2-2s)^n/n!$ for $x \in(0,1]$, and since $\int_0^1|\log x|^ndx=\int_0^\infty y^ne^{-y}dy=n!$ it follows that the function $s \mapsto |T'|^{1-s}$ is a holomorphic map from a neighbourhood of $1$ into $L^1([0,1])$. It is now a simple exercise to write $s \mapsto |T'|^{1-s}\xi_{s,0}$ as a power series in $(s-1)$ with coefficients in $L^1([0,1])$ and a nonzero radius of convergence, which proves the claim. For $s$ close to $1$ and $x \in (0,1)$ we may write
\begin{align*}(\mathcal{L}_{s,0}\xi_{s,0})(x)&=\sum_{n=1}^\infty \frac{1}{(x+n)^{2s}}\xi_{s,0}\left(\frac{1}{x+n}\right)\\
&=\sum_{n=1}^\infty \frac{1}{(x+n)^2}\left|T'\left(\frac{1}{x+n}\right)\right|^{1-s}\xi_{s,0}\left(\frac{1}{x+n}\right)=\mathcal{L}_{1,0}\left(|T'|^{1-s}\xi_{s,0}\right)(x)\end{align*}
 and hence the equation $\mathcal{L}_{1,0} \left(|T'|^{1-s} \xi_{s,0}\right)=\lambda(s,0)\xi_{s,0}$ is valid in $L^1([0,1])$. Differentiating this equation at $s=1$ we obtain
\[\mathcal{L}_{1,0}\left(-\log |T'|\xi + \xi^s_{1,0}\right) = \lambda_s(1,0)\xi+\xi^s_{1,0}\]
since $\xi_{1,0}=\xi$. Integration yields $\lambda_s(1,0)=-\int_0^1 \log|T'(x)|\xi(x)dx=-\mathfrak{h}(T)$ via Proposition \ref{TheOne}(ii) as desired.

We next claim that for each $m \geq 1$ the map $\omega \mapsto \exp(\omega \sum_{k=0}^{m-1}\psi \circ T^k)\xi_{1,\omega}$ is a holomorphic function from a small neighbourhood of $0 \in \mathbb{C}$ into $L^1([0,1])$. Since we may write $\xi_{1,\omega}=\sum_{n=0}^\infty \omega^n h_n$ for some sequence of functions $h_n \in H^\infty(\mathbb{D})$ with $\limsup_{n \to \infty} |h_n|_\infty^{1/n}<\infty$, it suffices to show that $\omega \mapsto \exp(\omega \sum_{k=0}^{m-1}\psi \circ T^k)$ is a holomorphic map into $L^1([0,1])$. Since $|\psi(x)|\leq |c|_\infty+|2\mu\log x|$ for all $x \in (0,1]$ it is clear that $\int_0^1|\psi(x)|^ndx \leq M^nn!$ for some constant $M>0$ depending on $c$, and by Proposition \ref{TheOne}(ii) there is a constant $K>0$ such that $\int_0^1 |\psi ( T^kx)|^n dx\leq K^k\int_0^1|\psi(x)|^ndx \leq K^kM^nn!$ for every $k \geq 0$. Using the 
generalised H\"older inequality
\begin{align*}\int_0^1\left|\sum_{k=0}^{m-1}\psi(T^kx)\right|^ndx &\leq \sum_{k_1,\ldots,k_n=0}^{m-1}\int_0^1 \prod_{i=0}^{n-1}\left|\psi(T^{k_i}x)\right| dx\\
&\leq\sum_{k_1,\ldots,k_n=0}^{m-1}\prod_{i=0}^{n-1}\left(\int_0^1 \left|\psi(T^{k_i}x)\right|^{n}dx\right)^{\frac{1}{n}}\leq m^nM^nK^{m} n!\end{align*}
and it follows that the power series $\sum_{n=0}^\infty \omega^n \left(\sum_{k=0}^{m-1}\psi \circ T^k\right)^n/n!$ with coefficients in $L^1([0,1])$ has nonzero radius of convergence, completing the proof of the claim.

Now, if $\omega$ is close to zero then since $\lambda(1,\omega)\xi_{1,\omega}=\mathcal{L}_{1,\omega}\xi_{1,\omega}$, using the definitions of $\mathcal{L}_{1,\omega}$ and $\psi$ we have
\begin{equation}\label{eq:instep}\lambda(1,\omega)\xi_{1,\omega}(x)=\sum_{n=1}^\infty\frac{1}{(x+n)^2}e^{\omega \psi\left(\frac{1}{x+n}\right)}f\left(\frac{1}{x+n}\right)=\mathcal{L}_{1,0}\left(e^{\omega\psi}\xi_{1,\omega}\right)(x)\end{equation}
for almost every $x \in [0,1]$. We claim that for all $n \geq 0$ the equation
\begin{equation}\label{eq:indhyp}\mathcal{L}_{1,0}^n\left(e^{\omega \sum_{k=0}^{n-1}\psi \circ T^k}\xi_{1,\omega}\right)=\lambda(1,\omega)^n\xi_{1,\omega}\end{equation}
in $L^1([0,1])$ is satisfied for all $\omega$ sufficiently close to zero. Clearly the claim holds for $n=0$. Given the validity of the claim for some $n\geq 0$, if $\omega$ is small enough that additionally $e^{\omega \sum_{k=0}^n \psi \circ T^k}\xi_{1,\omega} \in L^1([0,1])$ then we may deduce
\begin{align*}\mathcal{L}_{1,0}^{n+1}\left(e^{\omega \sum_{k=0}^{n}\psi \circ T^k}\right)&=\mathcal{L}_{1,0}^n\left(e^{\omega \sum_{k=0}^{n-1}\psi \circ T^k}\mathcal{L}_{1,0}(e^{\omega \psi}\xi_{1,\omega})\right)\\
&=\lambda(1,\omega)\mathcal{L}_{1,0}^n\left(e^{\omega \sum_{k=0}^{n-1}\psi \circ T^k}\xi_{1,\omega}\right)=\lambda(1,\omega)^{n+1}\xi_{1,\omega}\end{align*}
using Proposition \ref{TheOne}(ii), \eqref{eq:instep} and \eqref{eq:indhyp}, and the claim follows by induction.
 Differentiating \eqref{eq:indhyp} in $L^1([0,1])$ yields
\begin{eqnarray*}\lefteqn{\mathcal{L}_{1,0}^n\left(e^{\omega \sum_{k=0}^{n-1}\psi \circ T^k} \left(\sum_{k=0}^{n-1}(\psi \circ T^k)\xi_{1,\omega}+ \xi^\omega_{1,\omega}\right)\right)}\\
& =& n\lambda(1,\omega)^{n-1}\lambda_\omega(1,\omega)\xi_{1,\omega}+\lambda(1,\omega)^n\xi^\omega_{1,\omega}.\end{eqnarray*}
Setting $n=1$, $\omega=0$ and integrating over $[0,1]$ gives $\lambda_\omega(1,0)=\int_0^1\psi(x)\xi(x)dx=0$ as claimed. Differentiating a second time, setting $\omega=0$ and substituting $\xi_{1,0}=\xi$ and $\lambda_\omega(1,0)=0$ yields the expression
\[\mathcal{L}_{1,0}^n\left(\left(\sum_{k=0}^{n-1}\psi\circ T^k\right)^2\xi+ 2\left(\sum_{k=0}^{n-1}\psi\circ T^k\right)\xi^\omega_{1,0}+\xi^{\omega\omega}_{1,0}\right) = n\lambda_{\omega\omega}(1,0)\xi+\xi^{\omega\omega}_{1,0}.\]
Integrating, cancelling $\int_0^1\xi^{\omega\omega}_{1,0}(x)dx$ from both sides and dividing by $n$ we find that
\[\lambda_{\omega\omega}(1,0)=\frac{1}{n}\int_0^1\left(\left(\sum_{k=0}^{n-1}\psi(T^kx)\right)^2\xi(x)+ 2\left(\sum_{k=0}^{n-1}\psi(T^kx)\right)\xi^\omega_{1,0}(x)\right)dx\]
for all $n \geq 1$. Using Proposition \ref{TheOne}(iii) we have
\[\lim_{n \to \infty}\frac{1}{n}\sum_{k=0}^{n-1}\mathcal{L}_{1,0}^k\xi^\omega_{1,0}= \lim_{n \to \infty}\frac{1}{n}\sum_{k=0}^{n-1}\mathcal{P}_{1,0}^k\xi^\omega_{1,0}=\left(\int_0^1\xi^\omega_{1,0}(x)dx\right)\cdot \xi\]
in the uniform norm, and therefore
\begin{equation}\label{pupihed}\lim_{n \to \infty}\frac{1}{n}\int_0^1 \sum_{k=0}^{n-1}\psi(T^kx)\xi^\omega_{1,0}(x)dx =\lim_{n \to \infty}\int_0^1 \psi(x)\left(\frac{1}{n}\sum_{k=0}^{n-1}\mathcal{L}_{1,0}^k\xi^\omega_{1,0}\right)(x)dx =0\end{equation}
using Proposition \ref{TheOne}(ii) on each summand together with $\int_0^1\psi(x)\xi(x)dx=0$. Thus
\[\lambda_{\omega\omega}(1,0)=\lim_{n \to \infty}\frac{1}{n}\int_0^1 \left(\sum_{k=0}^{n-1}\psi(T^kx)\right)^2\xi(x)dx\geq 0\]
as required. It remains to prove that this quantity is nonzero. Define
\[\Psi(z):=\sum_{n=1}^\infty \frac{c(n)-\mu\log(z+n)}{(z+n)^2}\xi\left(\frac{1}{z+n}\right)\]
so that $\Psi \in H^\infty(\mathbb{D})$, and note that $\mathcal{L}_{1,0}(\psi\xi)=\Psi$ almost everywhere in $[0,1]$, so in particular $\int_0^1\Psi(x)dx=\int_0^1\psi(x)\xi(x)dx=0$. It follows that $\mathcal{L}_{1,0}^n\Psi=\mathcal{N}_{1,0}^n\Psi$ for every $n \geq 1$ and hence the sum $\chi:=\sum_{n=0}^\infty \xi^{-1}(\mathcal{L}_{1,0}^n\Psi)$ defines an element of $H^\infty(\mathbb{D})$. Define $\hat\psi:=\psi + \chi - \chi \circ T \in L^1([0,1])$ and observe that $\mathcal{L}_{1,0}(\hat\psi\xi)=\Psi + \mathcal{L}_{1,0}(\chi\xi) -\chi\xi= 0$ in $L^1([0,1])$. For each $n\geq 1$ we have
\[\left(\sum_{k=0}^{n-1}\psi \circ T^k\right)^2-\left(\sum_{k=0}^{n-1}\hat\psi\circ T^k\right)^2=2\left(\chi \circ T^n -\chi\right)\left(\sum_{k=0}^{n-1}\hat\psi\circ T^k\right) + \left(\chi \circ T^n -\chi\right)^2\]
in $L^1$. By an argument similar to \eqref{pupihed} we have $\frac{1}{n}\sum_{k=0}^{n-1}\hat\psi \circ T^k \to 0$ in $L^1([0,1])$, and since $\chi \circ T^N - \chi$ is bounded on $[0,1]$ independently of $n$ we may easily derive
\begin{equation}\label{pupihed2}\lambda_{\omega\omega}(1,0)=\lim_{n \to \infty}\frac{1}{n}\int_0^1 \left(\sum_{k=0}^{n-1}\hat\psi(T^kx)\right)^2\xi(x)dx.\end{equation}
Now, using Proposition \ref{TheOne}(ii) and the identity $\mathcal{L}_{1,0}\xi=\xi$, for $k>j \geq 0$ we have
\begin{align*}\int_0^1\hat\psi(T^kx)\hat\psi(T^jx)\xi(x)dx&=\int_0^1\mathcal{L}_{1,0}^k\left(\hat\psi\left(T^{k}x\right)\hat\psi(T^jx)\xi(x)\right)dx\\
&=\int_0^1\hat\psi(x)\mathcal{L}_{1,0}^{k-j}(\hat\psi\xi)(x)dx=0\end{align*}
since $\mathcal{L}_{1,0}(\hat\psi\xi)=0$. The off-diagonal terms of the squared sum in \eqref{pupihed2} thus vanish, leaving
\[\lambda_{\omega\omega}(1,0)=\lim_{n \to \infty}\frac{1}{n}\int_0^1 \left(\sum_{k=0}^{n-1}\hat\psi(T^kx)^2\right)\xi(x)dx=\int_0^1 \hat\psi(x)^2\xi(x)dx.\]
If this quantity were zero then we would have $\psi=\chi \circ T - \chi$ almost everywhere by the definition of $\hat\psi$. Since $c$ is bounded and the logarithm function is not, the function $\psi$ is not essentially bounded; but $\chi$ is bounded, so the equation $\psi=\chi \circ T - \chi$ a.e. is impossible and we conclude that $\lambda_{\omega\omega}(1,0)>0$ as claimed.
\end{proof}

\section{Part II: the Dirichlet series}\label{Thedir}
This short section comprises the following result:
\begin{proposition}\label{never-on-a-tuesday}
For each integer $p \geq 0$, the Dirichlet series
\[D_p(s):=\sum_{n=1}^\infty \frac{1}{n^{2s}}\sum_{(u,v) \in \Omega_n\setminus\Omega_{n-1}} \left(C(u/v)-\mu\log v\right)^p\]
is absolutely convergent in the half-plane $\Re(s)>1$. For each odd $p$, $D_p(s)$ is holomorphic on the line $\Re(s)=1$ with the possible exception of a pole at $s=1$ of order not greater than $(p+1)/2$. For each even $p$ there exist a holomorphic function $R$ defined in the half-plane $\Re(s) \geq 1$ and complex numbers $a_1,\ldots,a_{p/2}$ such that
\[D_p(s) = \frac{p!\sigma^{p}\mathfrak{f}_0}{2^p\mathfrak{h}(T)(s-1)^{1+\frac{p}{2}}}+\sum_{k=1}^{\frac{p}{2}} \frac{a_k}{(s-1)^k} + R(s)\]
when $\Re(s)>1$, where $\mathfrak{f}_0:=(\mathcal{F}_{1,0}\mathcal{P}_{1,0}\mathbf{1})(0)>0$ and $\sigma^2:=2\lambda_{\omega\omega}(1,0)/\mathfrak{h}(T)>0$.
\end{proposition}

\begin{proof}

For each $n \geq 0$ let us define $\Xi_n$ to be the set of all $u/v \in(0,1)$ which have a continued fraction expansion $u/v=[a_1,\ldots,a_{n+1}]$ with $a_{n+1}\geq 2$; equivalently, if we write $h_n(x):=1/(x+n)$ then $u/v \in \Xi_n$ if and only if $u/v=(h_{a_1} \circ h_{a_2}\circ\cdots\circ h_{a_{n+1}})(0)$ with $a_{n+1}\geq 2$. This composition of functions $h_k$ is itself a function of the form $x \mapsto (\alpha x + \beta)/(\gamma x + \delta)$ where $|\alpha\delta-\beta\gamma|=1$ and $\alpha,\beta,\gamma,\delta \in \mathbb{Z}$, so its derivative at $x=0$ is equal to $1/\delta^2 = 1/v^2$. Since each $u/v$ has a unique continued fraction expansion such that $a_{n+1}\geq 2$, using the definition of $\mathcal{F}_{s,\omega}$ and $\mathcal{L}_{s,\omega}$ we may write
\begin{align*}\left(\mathcal{F}_{s,\omega}\mathcal{L}_{s,\omega}^n\mathbf{1}\right)(0)&=\sum_{\substack{a_1,\ldots,a_n \geq 1\\a_{n+1}\geq 2}} e^{\omega\sum_{k=1}^{n+1}c(a_k) } \left|\left(h_{a_1} \circ \cdots \circ h_{a_{n+1}}\right)'(0)\right|^{s+\frac{1}{2}\mu\omega}  \\
&=\sum_{(u,v) \in \Xi_n} \frac{1}{v^{2s}}\exp(\omega(C(u/v)-\mu\log v)) \end{align*}
for each $n \geq 0$. Notice that when $\omega$ and $s$ are real this is a series of non-negative real numbers whose total is bounded by $|\mathcal{F}_{s,\omega}|_\infty |\mathcal{L}_{s,\omega}^n|_\infty$. Taking the sum over all integers $n \geq 0$ we obtain
\begin{equation}\label{had-to-do-an-overnight-redress}\sum_{n=1}^\infty \frac{1}{n^{2s}}\sum_{(u,v) \in \Omega_n\setminus \Omega_{n-1}} \exp\left(\omega\left(C(u/v)-\mu\log v\right)\right) = \sum_{k=0}^\infty \left(\mathcal{F}_{s,\omega}  \mathcal{L}^k_{s,\omega}\mathbf{1}\right)(0).\end{equation}
The second of these infinite series is well-defined for all $(s,\omega)$ such that $\rho(\mathcal{L}_{s,\omega})<1$. 
The first series is a convergent series of non-negative real numbers when $(s,\omega)$ are real and $\rho(\mathcal{L}_{s,\omega})<1$, and it is not difficult to see that this implies the absolute convergence of the series when $\rho\left(\mathcal{L}_{\Re(s),\Re(\omega)}\right)<1$. By Proposition \ref{they-said-i-got-too-much-flow} both series converge absolutely and \eqref{had-to-do-an-overnight-redress} is a valid equation when $\Re(s)>1$ and $\omega$ is sufficiently close to zero. When these conditions are met we may differentiate both sides with respect to $\omega$ to obtain for each integer $p \geq 0$
\begin{equation}\label{does-jude-law-have-a-vagina}\sum_{n=1}^\infty \frac{1}{n^{2s}}\sum_{(u,v) \in \Omega_n\setminus \Omega_{n-1}} \left(C\left(\frac{u}{v}\right)-\mu\log v\right)^p = \frac{\partial^p}{\partial\omega^p}\left(\sum_{k=0}^\infty \left(\mathcal{F}_{(\cdot,\cdot)}  \mathcal{L}^k_{(\cdot,\cdot)}\mathbf{1}\right)(0)\right)(s,0)\end{equation}
with the former series being absolutely convergent when $\Re(s)>1$. In view of Proposition \ref{they-said-i-got-too-much-flow} the second series in \eqref{had-to-do-an-overnight-redress} also convergent when $\Re(s)=1$, $s \neq 1$ and $\omega$ is sufficiently small, and so the right-hand side of \eqref{does-jude-law-have-a-vagina} is holomorphic in the region $\Re(s)\geq 1$, $s \neq 1$. In particular the series $D_p(s)$ admits an analytic continuation into that region as required.

The remainder of the proof is concerned with the analysis of $D_p(s)$ close to $s=1$. Let us fix the integer $p \geq 0$. Using Proposition \ref{TheOne}(iii) we may find an open set $\mathcal{V}\subset\mathbb{C}^2$ containing $(1,0)$ in which $\mathcal{L}_{s,\omega} = \lambda(s,\omega)\mathcal{P}_{s,\omega} \oplus \mathcal{N}_{s,\omega}$ where $\lambda(1,0)=1$, $\mathcal{P}_{s,\omega}$ is a projection of rank one and the spectral radius of $\mathcal{N}_{s,\omega}$ is strictly less than $1$. When $(s,\omega) \in \mathcal{V}$, $\Re(s)>1$ and $\omega$ is sufficiently small we have $|\lambda(s,\omega)|<1$ by Proposition \ref{they-said-i-got-too-much-flow} and in this case
\begin{align*} \sum_{k=0}^\infty \left(\mathcal{F}_{s,\omega}  \mathcal{L}^k_{s,\omega}\mathbf{1}\right)(0) &= \sum_{k=0}^\infty \lambda(s,\omega)^k\left(\mathcal{F}_{s,\omega}  \mathcal{P}_{s,\omega}\mathbf{1}\right)(0) + \left(\mathcal{F}_{s,\omega}  \mathcal{N}_{s,\omega}^k\mathbf{1}\right)(0)\\
 &=\frac{1}{1-\lambda(s,\omega)} (\mathcal{F}_{s,\omega}\mathcal{P}_{s,\omega}\mathbf{1})(0) + N(s,\omega),\end{align*}
say, where $N$ is holomorphic throughout $\mathcal{V}$. For $(s,0)\in\mathcal{V}$ with $\Re(s)>1$ the left-hand side of \eqref{does-jude-law-have-a-vagina} is therefore equal to
\[\frac{\partial^p}{\partial\omega^p}\left( \frac{1}{1-\lambda}\left(\mathcal{F}_{(\cdot,\cdot)}  \mathcal{P}_{(\cdot,\cdot)}\mathbf{1}\right)(0)\right)(s,0) +R_1(s)\]
 where $R_1$ extends holomorphically to a neighbourhood of $s=1$. Let us expand the first expression as
\[\sum_{n=0}^p\left(\begin{array}{c}n\\p\end{array}\right)\left(\frac{\partial^n}{\partial\omega^n} \left(\frac{1}{1-\lambda}\right)(s,0)\right)\left(\frac{\partial^{p-n}}{d\omega^{p-n}} \left((\mathcal{F}_{(\cdot,\cdot)}\mathcal{P}_{(\cdot,\cdot)}\mathbf{1})(0)\right)(s,0)\right).\]
Since $(\mathcal{F}_{s,\omega}\mathcal{P}_{s,\omega}\mathbf{1})(0)$ is holomorphic throughout $\mathcal{V}$ its partial derivatives are holomorphic also, and hence poles may only enter the above expression via the partial derivatives of $1/(1-\lambda)$. We will show that the term corresponding to $n=p$ contributes a pole of order $\lfloor 1+p/2\rfloor$, and the remaining terms can only contribute poles of lower order. Using Fa\`a di Bruno's formula (see e.g. \cite{Rio}),
\begin{equation}\label{fa}\frac{\partial^n}{\partial\omega^n} \left(\frac{1}{1-\lambda}\right)(s,0) = \sum_{\substack{k_1+2k_2+\ldots+nk_n=n\\k_1,\ldots,k_n \geq 0}}\frac{n!(\sum_{\ell=1}^n k_\ell)!}{(1-\lambda(s,0))^{1+\sum_{\ell=1}^n k_\ell}}\prod_{\ell=1}^n  \frac{\lambda_\omega^{(\ell)}(s,0)^{k_\ell}}{k_\ell! (\ell!)^{k_\ell}}\end{equation}
We consider separately the contribution from each summand.

Recall from Proposition \ref{a-owl} that $\lambda_\omega(s,0)$ has a zero at $s=1$ with degree equal to some positive integer $d$, whereas $1-\lambda_(s,0)$ has a simple zero at $s=1$. If $k_1 \geq 2$ then $\sum_{\ell=2}^n k_\ell \leq \frac{1}{2}(n-k_1)\leq n/2 -1$ and so the function $\lambda_\omega(s,0)^{k_1} (1-\lambda(s,0))^{-1-\sum_\ell k_\ell}$ has a pole of order at most $n/2$ at $s=1$. If $k_i \geq 1$ for some $i \geq 3$, then $\sum_{\ell=2}^n k_\ell \leq \frac{1}{2}(n-k_i)  \leq (n-1)/2$ and hence $\lambda_\omega(s,0)^{k_1}(1-\lambda(s,0))^{-1-\sum_\ell k_\ell}$ has a pole of order not greater than $(n+1)/2$. The only remaining case is that in which $k_2 = \lfloor n/2 \rfloor$ and $k_1 = n -2k_2$, and in this case the order of the pole of $\lambda_\omega(s,0)^{k_1}(1-\lambda(s,0))^{-1-\sum_\ell k_\ell}$ at $s=1$ is precisely $1+\lfloor n/2\rfloor+(1-d)k_1$. It follows that when $n$ is odd the expression \eqref{fa} has a pole at $s=1$ of order not greater than $(n+1)/2$, and when $n$ is even it has a pole of order precisely $1+n/2$ at $s=1$ which arises solely from the summand with $k_2=n/2$ and $k_1=0$. Summing \eqref{fa} over $n$ in the range $0$ to $p$, it follows that when $p$ is odd $D_p$ has a pole of order at most $(p+1)/2$ at $s=1$, which completes the proof of the proposition for odd $p$. In the case where $p$ is even the above analysis shows that we may write 
\[\frac{\partial^p}{\partial\omega^p}\left(\frac{1}{1-\lambda} (\mathcal{F}_{(\cdot,\cdot)}\mathcal{P}_{(\cdot,\cdot)}\mathbf{1})(0)\right)(s,0) =\frac{p! \lambda_{\omega\omega}(s,0)^{p/2}}{2^{p/2}(1-\lambda(s,0))^{1+p/2}}(\mathcal{F}_{s,0}\mathcal{P}_{s,0}\mathbf{1})(0) + R_2(s)\]
for all $s \neq 1$ such that $(s,0)\in\mathcal{V}$, where $R_2$ has a pole at $s=1$ of order not greater than $p/2$ and otherwise is holomorphic at $s$ for $(s,0)\in\mathcal{V}$. Let $\mathfrak{f}_0:=(\mathcal{F}_{1,0}\mathcal{P}_{1,0}\mathbf{1})(0)$, which is positive since $\mathcal{P}_{1,0}\mathbf{1}=\xi$ and $\mathcal{F}_{1,0}$ preserves the set of functions which are positive on $[0,1]$. Since $\lambda(1,0)=1$, $\lambda_{\omega\omega}(1,0)=\frac{1}{2}\sigma^2\mathfrak{h}(T)$, and $|\lambda(s,0)|<1$ when $\Re(s) \geq 1$ and $s \neq 1$, the functions $A(s):=(1-s)/(1-\lambda(s,0))$, $B(s):=(\lambda_{\omega\omega}(s,0)-\frac{1}{2}\sigma^2\mathfrak{h}(T))/(1-s)$ and $C(s):=((\mathcal{F}_{s,0}\mathcal{P}_{s,0}(\mathbf{1})(0)-\mathfrak{f}_0)/(1-s)$ are holomorphic in the region $\Re(s) \geq 1$. Substituting $(1-\lambda(s,0))^{-1}=(1-s)^{-1}A(s)$, $\lambda_{\omega\omega}(s,0) = \frac{1}{2}\sigma^2\mathfrak{h}(T)+(1-s)B(s)$ and $(\mathcal{F}_{s,0}\mathcal{P}_{s,0}\mathbf{1})(0)=\mathfrak{f}_0+(1-s)C(s)$ into the previous expression we obtain
\[\frac{\partial^p}{\partial^p\omega}\left(\frac{1}{1-\lambda} (\mathcal{F}_{(\cdot,\cdot)}\mathcal{P}_{(\cdot,\cdot)}\mathbf{1})(0)\right)(s,0) =\frac{p!\sigma^p\mathfrak{h}(T)^{p/2}A(s)^{1+p/2}\mathfrak{f}_0}{2^{p}(1-s)^{1+p/2}} + R_3(s)\]
for all $s$ sufficiently close to $1$, where $R_3$ extends holomorphically to the region $\Re(s) \geq 1$ with the possible exception of a pole at $s=1$ with order not greater than $p/2$. In view of Proposition \ref{a-owl} we have $A(1)=\lambda_s(1,0)^{-1}=-1/\mathfrak{h}(T)$, and so this may be further rewritten as 
\[\frac{\partial^p}{\partial\omega^p}\left(\frac{1}{1-\lambda} (\mathcal{F}_{(\cdot,\cdot)}\mathcal{P}_{(\cdot,\cdot)}\mathbf{1})(0)\right)(s,0) =\frac{p!\sigma^p\mathfrak{f}_0}{2^{p}\mathfrak{h}(T)(s-1)^{1+p/2}} + R_4(s)\]
where $R_4$ has the same properties as $R_3$. The proof is complete.
 \end{proof}

\section{Part III: the Tauberian argument}\label{Thetau}
In this section we apply the properties of the Dirichlet series $D_p(s)$ established in Proposition \ref{never-on-a-tuesday} to prove Theorem \ref{qfwfq}. We give the full details of the proof only in the case of the distribution of costs on $\Omega_n$; the case of the distribution on $\tilde{\Omega}_n$ can be handled by simple modifications which are indicated at the end of the section. By the classical method of moments, to prove Theorem \ref{qfwfq} it is sufficient to prove that for every $p \geq 1$ the sequence of $p^{\mathrm{th}}$ moments
\[\frac{1}{\#\Omega_n}\sum_{(u,v) \in \Omega_n} \left(\frac{C(u/v)-\mu\log n}{\sqrt{\log n}}\right)^p\]
converges to the limit
\[\frac{1}{\sigma\sqrt{2\pi}}\int_{-\infty}^\infty t^pe^{-\frac{t^2}{2\sigma^2}}dt=\Bigg\{\begin{array}{cl}0& \text{if }p\text{ is odd}\\
\sigma^p (p-1)!!& \text{if }p\text{ is even,}\end{array}\]
see for example \cite[\S30]{Bill}. We will begin by calculating the asymptotics of the related sequence $\sum_{(u,v)\in\Omega_n} (C(u/v)-\mu\log v)^p$ as $n \to \infty$ for each $p \geq 0$, and then use this to derive the corresponding asymptotics for the sequences of moments. We require the following Tauberian theorem due to H. Delange \cite[Th. III]{Delange}:
\begin{theorem}[Delange]\label{you-go-tell-dr-dre-that-man-will-fuck-you-up}
Let $(a_n)_{n \geq 1}$ be a sequence of non-negative real numbers. Suppose that for all $s$ in the half-plane $\Re(s)>a$ the series $D(s):= \sum_{n=1}^\infty a_nn^{-s}$ is absolutely convergent and satisfies
\[D(s) = \frac{g(s)}{(s-a)^\omega}+\sum_{j=1}^m \frac{g_j(s)}{(s-a)^{\lambda_j}} + h(s)\]
 where $\omega>\lambda_j>0$ for every $j$, the functions $g_i$, $g$ and $h$ are holomorphic in the region $\Re(s) \geq a$, and $g(a)>0$. Then
\begin{equation}\label{come-on-dickhead}\lim_{N \to \infty}\frac{1}{N^a(\log N)^{\omega-1}}\sum_{n=1}^N a_n = \frac{g(a)}{a\Gamma(\omega)}.\end{equation}
\end{theorem}
{\it{Remark.}} Delange's statement includes some additional parameters $\mu_j$ all of which we set to $0$ in the above statement. Since we will apply the result to series of the form $\sum_{n=1}^\infty a_nn^{-2s}$ we obtain an additional factor $2^{\omega-1}$ in the numerator of the limit in \eqref{come-on-dickhead} which arises from the change of variable.

Let us write $D^{(k)}_p$ for the $k^{\mathrm{th}}$ derivative of $D_p$ with respect to $s$. In the case $p=0$ the combination of Proposition \ref{never-on-a-tuesday} with Theorem \ref{you-go-tell-dr-dre-that-man-will-fuck-you-up} yields $\lim_{n \to \infty} \#\Omega_n/n^2=\mathfrak{f}_0/\mathfrak{h}(T)$. For even $p>0$ the same approach yields
\[\lim_{n \to \infty} \frac{1}{n^2(\log n)^{p/2}} \sum_{(u,v)\in\Omega_n} \left(C(u/v) - \mu \log v\right)^p = \frac{p!\sigma^p\mathfrak{f}_0}{2^{p/2}(p/2)!\mathfrak{h}(T)}\]
and therefore
\begin{equation}\label{christmas-shitter}\lim_{n \to \infty} \frac{1}{\#\Omega_n (\log n)^{p/2}} \sum_{(u,v)\in\Omega_n} \left(C(u/v) - \mu \log v\right)^p = \frac{p!\sigma^p}{2^{p/2}(p/2)!}=(p-1)!! \sigma^p.\end{equation}
When $p \geq 1$ is odd the coefficients of the series $D_p$ fail to be non-negative and a more delicate argument is required, which comprises much of the remainder of this section. We adopt an argument due to H.-K. Hwang and S. Janson \cite{HJ,HJcorr}. Let us fix odd $p \geq 1$ and define auxiliary Dirichlet series,
\begin{eqnarray}
\notag\mathcal{D}_1(s)&:=&\sum_{n=1}^\infty \frac{1}{n^{2s}}\sum_{(u,v) \in \Omega_n\setminus \Omega_{n-1}} \left(\left(C\left(\frac{u}{v}\right) - \mu\log v\right)^{2p}+(\log v)^p\right),\\
\notag\mathcal{D}_2(s)&:=&\sum_{n=1}^\infty \frac{1}{n^{2s}}\sum_{(u,v)\in\Omega_n\setminus \Omega_{n-1}} \left(\left(C\left(\frac{u}{v}\right) -  \mu\log v\right)^{p}+(\log v)^{\frac{p}{2}}\right)^2,\\
\label{i-can-smell-your-fucking-hands}\mathcal{D}_3(s)&:=&\sum_{n=1}^\infty \frac{1}{n^{2s}}\sum_{(u,v)\in\Omega_n\setminus \Omega_{n-1}} \left(\log v\right)^{\frac{p}{2}}\left(C\left(\frac{u}{v}\right) -  \mu\log v\right)^{p},\end{eqnarray}
where we note that $\mathcal{D}_3=\frac{1}{2}(\mathcal{D}_2-\mathcal{D}_1)$. The first of these series  is easily seen to satisfy
$\mathcal{D}_1(s)=D_{2p}(s)-\frac{1}{2^p}D_0^{(p)}(s)$, and since $D_{2p}$ and $D_0$ converge absolutely in the region $\Re(s)>1$ the series $\mathcal{D}_1$ converges absolutely in that region also. Since the terms of $\mathcal{D}_3(s)$ are bounded in absolute value by those of $(-2)^{-\lceil p/2\rceil}(D_p^{(\lceil p/2 \rceil)}(\Re(s))+D_0^{\lceil p/2\rceil}(\Re(s)))$ it follows that $\mathcal{D}_3(s)$ converges absolutely for $\Re(s)>1$, and we deduce that $\mathcal{D}_2=\mathcal{D}_1+2\mathcal{D}_3$ has the same property.

Using the properties of $D_{2p}$ and $D_0$ described in Proposition \ref{never-on-a-tuesday} we may write
\[\mathcal{D}_1(s)=D_{2p}(s)-\frac{1}{2^p}D_0^{(p)}(s)=\frac{g(s)}{(s-1)^{1+p}}+f(s)\]
where $g$ and $f$ are holomorphic in the region $\Re(s) \geq 1$ and $g(1)$ is positive (its precise value is unimportant).
Using Theorem \ref{you-go-tell-dr-dre-that-man-will-fuck-you-up} it follows that
\[\lim_{n \to \infty} \frac{1}{n^2(\log n)^p} \sum_{(u,v) \in \Omega_n} \left(\left(C(u/v) - \mu\log v\right)^{2p}+(\log v)^p\right) = \frac{2^p g(1)}{p!}.\]
We will show that $\mathcal{D}_2$ has a singularity at $s=1$ with the same order and leading term as $\mathcal{D}_1$ and hence that its partial sums have identical asymptotic behaviour. To achieve this we shall bound the singularity of the difference $\mathcal{D}_3$. We have
 \[\left(-\frac{1}{2}\right)^{\frac{p+1}{2}}D_p^{\left(\frac{p+1}{2}\right)}(s)=\sum_{n=1}^\infty \frac{1}{n^{2s}}\sum_{(u,v) \in \Omega_n\setminus \Omega_{n-1}} (\log v)^{\frac{p+1}{2}}\left(C\left(\frac{u}{v}\right)-\mu\log v\right)^p\]
for $\Re(s)>1$. Noting that $v=n$ when $(u,v) \in \Omega_n\setminus \Omega_{n-1}$ and using the identity $\int_0^\infty n^{-2t}t^{-\frac{1}{2}}dt=\sqrt{\pi/(2\log n)}$ we may rewrite \eqref{i-can-smell-your-fucking-hands} as
\begin{align*}\mathcal{D}_3(s)&=\sqrt{\frac{2}{\pi}}\sum_{n=1}^\infty \frac{1}{n^{2s}}\sum_{(u,v)\in\Omega_n\setminus \Omega_{n-1}} (\log v)^{\frac{p+1}{2}}\left(C\left(\frac{u}{v}\right)-\mu\log v\right)^p \int_0^\infty \frac{dt}{n^{-2t}\sqrt{t}}\\
&=\frac{(-1)^{\frac{p+1}{2}}}{\sqrt{2^p\pi}}\int_0^\infty \frac{D_p^{((p+1)/2)}(s+t)}{\sqrt{t}}dt.\end{align*}
We have $|D_p^{((p+1)/2)}(s)|=O(2^{-\Re(2s)})$ uniformly in $\Im(s)$ in the limit as $\Re(s) \to \infty$ since $D_p^{((p+1)/2)}$ is a Dirichlet series with zero constant term. It follows that when $D_p^{((p+1)/2)}$ is bounded and holomorphic in a strip of the form $\{s \colon \Re(s)>\alpha \text{ and }\beta <\Im(s) < \gamma\}$ the integral $\int_0^\infty D_p^{((p+1)/2)}(s+t)t^{-1/2}dt$ converges absolutely uniformly in the same strip. We deduce that $\mathcal{D}_3$ is holomorphic in the region $\Re(s) \geq 1$ except for the point $s=1$.

By Proposition \ref{never-on-a-tuesday} $D_p$ has a pole at $s=1$ with order at most $(p+1)/2$, and therefore $D_p^{((p+1)/2)}$ has a pole at $s=1$ of order at most $p+1$. It follows that we may find an open set $U$ of the form $U=\{s \colon \Re(s) >1-\varepsilon, |\Im(s)|<\delta\}$, a bounded holomorphic function $R$ defined on $U$, and complex numbers $a_1,\ldots,a_{\ell}$ such that
\[D_p^{((p+1)/2)}(s)=\sum_{k=1}^{\ell}\frac{a_k}{(s-1)^k} + R(s)\]
for all $s \in U$, where $\ell \leq p+1$. We have $|R(s)|=O(\Re(s)^{-1})$ as $\Re(s) \to \infty$ uniformly throughout $U$  in view of the corresponding bound on $|D_p^{((p+1)/2)}(s)|$ and hence the integral $\int_0^\infty R(s+t)t^{-1/2}dt$ converges absolutely uniformly in this region. Since
\[\int_0^\infty \frac{1}{(s+t-1)^k\sqrt{t}}dt = (s-1)^{\frac{1}{2}-k}\int_0^\infty \frac{du}{(1+u^2)^k} = \frac{\pi(2k-2)!}{2^{2k-1}((k-1)!)^2}(s-1)^{\frac{1}{2}-k}\]
for integers $k \geq 1$ it follows that for $s \in U$ we have
\[\mathcal{D}_3(s)=(-1)^{\frac{p+1}{2}}\sum_{k=1}^{\ell} \frac{\sqrt{\pi}(2k-2)!}{2^{2k-\frac{p}{2}-1}((k-1)!)^2}a_k(s-1)^{\frac{1}{2}-k} + \hat{R}(s)\]
where $\hat{R}$ is another holomorphic function. Taking the Laurent series of $\mathcal{D}_1$ around its pole at $s=1$ it follows that $\mathcal{D}_2=\mathcal{D}_1+2\mathcal{D}_3$ may be written in the form
\[\mathcal{D}_2(s)=\sum_{k=1}^{p+1} \frac{f_k}{(s-1)^k} + \frac{g_k}{(s-1)^{k-\frac{1}{2}}} + r(s)\]
where $f_k$, $g_k$ and $r$ are all holomorphic in the region $\Re(s) \geq 1$, and $f_{p+1}(1)=g(1)$. Applying Theorem \ref{you-go-tell-dr-dre-that-man-will-fuck-you-up} to $\mathcal{D}_2$ we obtain
\[\lim_{n \to \infty} \frac{1}{n^2(\log n)^p} \sum_{(u,v) \in \Omega_n} \left(\left(C(u/v) - \mu\log v\right)^p+(\log v)^{p/2}\right)^2 = \frac{2^p g(1)}{p!}\]
and by subtracting the corresponding expression for the partial sums of $\mathcal{D}_1$
\begin{equation}\label{sweetie-man}\lim_{n \to \infty} \frac{1}{n^2(\log n)^p} \sum_{(u,v) \in \Omega_n} (\log v)^{p/2}\left(C(u/v) - \mu\log v\right)^{p}=0.\end{equation}
We next show that \eqref{sweetie-man} remains valid if the factor $(\log v)^{p/2}$ is removed from inside the summation and a factor $(\log n)^{p/2}$ is removed from the denominator outside the summation. Following \cite{HJ}, for each $t>0$ let us define $\varphi(t):=(\log t)^{-p/2}$ and $\Pi(t):=\sum_{(u,v) \in \Omega_{\lfloor t \rfloor}} \left(\log v \right)^{p/2}\left(C(u/v) - \mu\log v\right)^p$. Using \eqref{sweetie-man} we may find $\tau>2$ such that $|\Pi(t)|< 2t^{2}(\log t)^{p}/p$ for all $t>\tau$ and hence for all $n>\tau$
\begin{align*}\left|\int_2^n \Pi(t)\varphi'(t)dt\right| &\leq \left|\int_2^\tau  \Pi(t)\varphi'(t)dt\right| + \left|\int_\tau^n \Pi(t)\varphi'(t)dt \right|\\
& \leq K+ \int_2^n t(\log t)^{\frac{p}{2}-1} dt<K+n^2(\log n)^{\frac{p}{2}-1},\end{align*}
say. It follows that
\begin{equation}\label{nonce-nuggets}\lim_{n \to \infty} \frac{1}{n^2(\log n)^{p/2}}\int_2^n \Pi(t) \varphi'(t)dt=0.\end{equation}
Noting that $(\log v)^{p/2}(C(u/v)-\mu\log v)=0$ when $v=1$, for each $n \geq 3$ we obtain
\begin{align*}\int_2^{n} \Pi(t)\varphi'(t)dt
 &= \sum_{(u,v) \in \Omega_{n-1}} \left(\log v \right)^{\frac{p}{2}}\left(C\left(\frac{u}{v}\right) - \mu\log v\right)^p \int_v^{n} \varphi'(t)dt\\
&= \sum_{(u,v) \in \Omega_{n-1}} \left(\log v \right)^{\frac{p}{2}}\left(C\left(\frac{u}{v}\right) - \mu\log v\right)^p \left((\log n)^{-\frac{p}{2}}-(\log v)^{-\frac{p}{2}}\right)\\
&=\frac{1}{(\log n)^{p/2}} \Pi(n-1) - \sum_{(u,v) \in \Omega_{n-1}}\left(C(u/v) - \mu\log v\right)^p\end{align*}
and in combination with \eqref{sweetie-man} and \eqref{nonce-nuggets} it is not difficult to see that this implies
\begin{equation}\label{if-man-say-im-a-ting-then-im-a-ting}\lim_{n \to \infty} \frac{1}{n^2(\log n)^{p/2}} \sum_{(u,v) \in \Omega_n} \left(C(u/v) - \mu \log v\right)^p =0.\end{equation}
To complete the proof of the theorem we must show that for all $p \geq 1$,
\begin{equation}\label{bell}\lim_{n \to \infty} \frac{1}{\#\Omega_n} \sum_{(u,v) \in \Omega_n} \left(\frac{C(u/v) - \mu \log n}{\sqrt{\log n}}\right)^p =\Bigg\{\begin{array}{cl}0& \text{if }p\text{ is odd}\\
\sigma^p (p-1)!!& \text{if }p\text{ is even,}\end{array}\end{equation}
which differs from \eqref{christmas-shitter} and \eqref{if-man-say-im-a-ting-then-im-a-ting} in that the term $\mu\log v$ is replaced with $\mu\log n$. To simplify the notation in the sequel let us write $a_{n,p}:=\sum_{(u,v) \in \Omega_n} (C(u/v)-\mu\log v)^p$ and $b_{n,p}:=\sum_{(u,v)\in \Omega_n} (C(u/v)-\mu\log n)^p$ for each $n,p \geq 1$.

In the case where $p$ is even, let us choose $K>0$ such that $(\mu\log x)^p \leq K^p\sqrt{x}$ for all $x \geq 1$. Using the reverse triangle inequality for the $p$-norm we obtain
\begin{equation}\label{pollo-con-pesto-for-the-lady}\left|\left(a_{n,p}\right)^{\frac{1}{p}}-\left(b_{n,p}\right)^{\frac{1}{p}}\right|\leq \left(\sum_{(u,v) \in \Omega_n} \mu^p \left|\log n - \log v\right|^p\right)^{\frac{1}{p}}\leq K\left(\sum_{(u,v) \in \Omega_n} \sqrt{\frac{n}{v}}\right)^{\frac{1}{p}}.\end{equation}
Applying Theorem \ref{you-go-tell-dr-dre-that-man-will-fuck-you-up} to the series $D_0(s+\frac{1}{4})=\sum_{n=1}^\infty n^{-2s}(\sum_{(u,v) \in \Omega_n\setminus \Omega_{n-1}} v^{-1/2})$ yields the asymptotic
\begin{equation}\label{spicy-sausage-penne-for-the-man}\lim_{n \to \infty} \frac{1}{\#\Omega_n}\sum_{(u,v) \in \Omega_n} \sqrt{\frac{n}{v}}=\left(\frac{\mathfrak{h}(T)}{\mathfrak{f}_0}\right)\left(\lim_{n \to \infty} \frac{1}{n\sqrt{n}}\sum_{(u,v) \in \Omega_n} \frac{1}{\sqrt{v}}\right)=\frac{2}{3},\end{equation}
and in view of \eqref{christmas-shitter} it follows that
\begin{equation}\label{new-man}\lim_{n \to \infty} \left(\frac{b_{n,p}}{\#\Omega_n(\log n)^{p/2}}\right)^\frac{1}{p} =\lim_{n \to \infty} \left(\frac{ a_{n,p}}{\#\Omega_n(\log n)^{p/2}}\right)^\frac{1}{p} =\sigma \left((p-1)!!\right)^{\frac{1}{p}}\end{equation}
which yields \eqref{bell} for even $p$.

We now consider odd $p$. The difference $|a_{n,p}-b_{n,p}|$ may be bounded by
\begin{equation}\label{shut-up-college-phd-kid}p\mu \sum_{(u,v) \in\Omega_n}\log\left(\frac{n}{v}\right)\left(\left(C\left(\frac{u}{v}\right)-\mu\log v\right)^{p-1}+\left(C\left(\frac{u}{v}\right)-\mu\log n\right)^{p-1}\right)\end{equation}
using the elementary inequality $\left|x^p-y^p\right| \leq p|x-y|(|x|^{p-1}+|y|^{p-1})$ for real $x,y$.
Let $\varepsilon_n:=(\log n)^{-2/3}$ for each $n> 1$ and consider separately the summation in \eqref{shut-up-college-phd-kid} over pairs $(u,v)$ such that $v \leq n^{1-\varepsilon_n}$ and over pairs $(u,v)$ such that $v > n^{1-\varepsilon_n}$. The first sum is simply a sum over $(u,v) \in \Omega_{\lfloor n^{1-\varepsilon_n}\rfloor}$ and so using the bound $\log (n/v) \leq \log n$ together with \eqref{new-man} we obtain
\[p\mu (\log n)\left(a_{\lfloor n^{1-\varepsilon_n}\rfloor,p-1} + b_{\lfloor n^{1-\varepsilon_n}\rfloor,p-1}\right)=O\left((\log n)^{\frac{p+1}{2}} n^{2-2\varepsilon_n}\right)=o(n^2).\]
In evaluating the second sum we may use the estimate $\log (n/v) <\varepsilon_n \log n =(\log n)^{1/3}$ to obtain the upper bound
\[p\mu(\log n)^{1/3} (a_{n,p-1}+b_{n,p-1})=O(n^2(\log n)^{\frac{p}{2} - \frac{1}{6}})=o(n^2 (\log n)^{p/2}).\]
We conclude that $|a_{n,p}-b_{n,p}|=o(n^2(\log n)^{p/2})$ which allows us to deduce \eqref{bell}, completing the proof for the distributions on $\Omega_n$.

In the case of the distribution on $\tilde{\Omega}_n$, pointwise multiplication of Dirichlet series yields
\begin{align*}\zeta(2s)D_p(s)&=\sum_{n=1}^\infty \frac{1}{n^{2s}}\sum_{d | n}\left(\sum_{(u,v) \in \Omega_d\setminus\Omega_{d-1}} \left(C(u/v)-\mu\log v\right)^p\right)\\
&=\sum_{n=1}^\infty \frac{1}{n^{2s}}\sum_{(u,v) \in \tilde{\Omega}_n\setminus\tilde{\Omega}_{n-1}} \left(C(u/v)-\mu\log \left(\frac{v}{\mathrm{gcd}(u,v)}\right)\right)^p\end{align*}
for $\Re(s)>1$, where $\zeta$ is the Riemann zeta function. To treat this case we consider modified Dirichlet series $\tilde{\mathcal{D}}_1$,$\tilde{\mathcal{D}}_2$,$\tilde{\mathcal{D}}_3$ which are derived from the corresponding series $\mathcal{D}_i$ by replacing $(C(u/v)-\mu\log v)$ with $(C(u/v)-\mu\log(v/\mathrm{gcd}(u,v)))$ and leaving the terms $(\log v)^p$, $(\log v)^{p/2}$ unchanged. These alternative Dirichlet series are analysed by substituting $\zeta(2s)D_p(s)$ for $D_p(s)$ throughout the preceding arguments in the obvious fashion until we obtain
\[\lim_{n \to \infty} \frac{1}{\#\tilde{\Omega}_n} \sum_{(u,v) \in \tilde{\Omega}_n} \left(\frac{C(u/v) - \mu \log \left(\frac{v}{\mathrm{gcd}(u,v)}\right)}{\sqrt{\log n}}\right)^p =\Bigg\{\begin{array}{cl}0& \text{if }p\text{ is odd}\\
\sigma^p (p-1)!!& \text{if }p\text{ is even.}\end{array}\]
To derive the analogue of \eqref{bell} for even $p$ we follow the model of \eqref{pollo-con-pesto-for-the-lady}, noting that the Dirichlet convolution
\[\zeta(2s)D_0\left(s+\frac{1}{4}\right)=\sum_{n=1}^\infty n^{-2s}\sum_{(u,v)\in\tilde{\Omega}_n\setminus \tilde{\Omega}_{n-1}} \sqrt{\frac{\mathrm{gcd}(u,v)}{v}}\]
yields the estimate
\[\sum_{(u,v) \in \tilde{\Omega}_n} \sqrt{\frac{n\cdot\mathrm{gcd}(u,v)}{v}}=O\left(n^2\right)\]
in lieu of \eqref{spicy-sausage-penne-for-the-man}, which suffices to complete the argument for even $p$. The argument for odd $p$ proceeds unchanged.

\section{Acknowledgments}
This research was supported by EPSRC grant EP/L026953/1.

\bibliographystyle{amsplain}
\bibliography{euclid}
\end{document}